\documentclass[11pt]{amsart}

\makeatletter
\@namedef{subjclassname@2020}{\emph{2020} Mathematics Subject Classification}
\makeatother

\usepackage[margin=1in]{geometry} %margins
\usepackage{amsmath, amsthm, amsfonts, amssymb} %ams packages
\usepackage{mathrsfs} %script letters
\usepackage{enumerate} %custom lists
\usepackage{xcolor} %allow colorful text
\usepackage{bbm} %additional blackboard bold characters
\usepackage{hyperref} %include hyperlinks
\usepackage{bm} %bold math

\usepackage{chngcntr}
\counterwithin{equation}{section} %numbers equations by section

\usepackage{thmtools, thm-restate} %packages for recalling theorems

\theoremstyle{theorem}
	\newtheorem{theorem}{Theorem}[section]
	
	\newtheorem{proposition}[theorem]{Proposition}
	\newtheorem{corollary}[theorem]{Corollary}
\theoremstyle{definition}
	\newtheorem{remark}[theorem]{Remark}

\newcommand{\N}{\mathbb{N}}
\newcommand{\Z}{\mathbb{Z}}

\newcommand{\R}{\mathbb{R}}

\newcommand{\T}{\mathbb{T}}

\newcommand{\K}{\mathbb{K}}

\newcommand{\F}{\mathbb{F}}

\newcommand{\mZ}{\mathcal{Z}}
\newcommand{\mQ}{\mathcal{Q}}
\newcommand{\mR}{\mathcal{R}}
\newcommand{\mT}{\mathcal{T}}

\newcommand{\mK}{\mathcal{K}}

\renewcommand{\hat}{\widehat}

\newcommand{\seminorm}[2]{{\left\vert\kern-0.25ex\left\vert\kern-0.25ex\left\vert #2 
    \right\vert\kern-0.25ex\right\vert\kern-0.25ex\right\vert}_{#1}}

\title{A Weyl equidistribution theorem over function fields}

\author{Ethan Ackelsberg}

\date{\today}



\begin{document}

\maketitle

\begin{abstract}
	A classical theorem of Weyl states that any polynomial with an irrational coefficient other than the constant term is uniformly distributed mod 1.
	We prove a new function field analogue of this statement, confirming a conjecture of L\^{e}, Liu, and Wooley.
\end{abstract}

%%%%%%%%%%%%%%%%%%%%%%%%%%%%%%%%%%%%%%%%%%%%%%%%%%%%%%%%
% ---- INTRODUCTION ---- %
%%%%%%%%%%%%%%%%%%%%%%%%%%%%%%%%%%%%%%%%%%%%%%%%%%%%%%%%

\section{Introduction}

The purpose of this short note is to provide a positive resolution to a conjecture of L\^{e}, Liu, and Wooley (\cite[Conjecture 1.3]{llw}) concerning equidistribution of polynomial sequences over function fields.

We briefly review Weyl's classical equidistribution theorem as motivation for the function field analogue.
A sequence of real numbers $(x_n)_{n \in \N}$ is \emph{uniformly distributed mod 1} if for every interval $(a,b) \subseteq [0,1]$,
\begin{equation*}
	\frac{\left| \left\{ 1 \le n \le N : \{x_n\} \in (a,b) \right\} \right|}{N} \to b-a,
\end{equation*}
where $\{x_n\}$ denotes the fractional part of $x_n$.
In other words, the amount of time the sequence $(x_n)_{n \in \N}$ spends in any given interval is proportional to the length of the interval.
The study of uniform distribution was initiated by Weyl \cite{weyl}, who proved the following remarkable theorem.

\begin{theorem}[{\cite[Satz 9]{weyl}}] \label{thm: Weyl}
	Let $P(n) = \alpha_d n^d + \ldots + \alpha_1 n + \alpha_0 \in \R[n]$.
	If at least one of $\alpha_1, \ldots, \alpha_d$ is irrational, then $(P(n))_{n \in \N}$ is uniformly distributed mod 1.
\end{theorem}

Let us now set our notation and terminology for the function field setting.
Fix a finite field $\F_q$ of characteristic $p$.
We introduce the suggestive notation $\mZ$ for the polynomial ring $\F_q[t]$ and $\mQ$ for the field of rational functions $\F_q(t)$.
The field $\mQ$ has an absolute value $\left| \frac{u}{v} \right| = q^{\deg{u} - \deg{v}}$ with the convention that $\deg{0} = - \infty$.
The completion of $\mQ$ with respect to this absolute value is (isomorphic to) the field of formal Laurent series
\begin{equation*}
	\mR = \F_q((t^{-1})) = \left\{ \sum_{n=-\infty}^N c_n t^n : N \in \Z, c_n \in \F_q \right\}.
\end{equation*}
The quotient space $\mR/\mZ$ is a compact group, which we denote by $\mT$.
In analogy with the well-known isomorphism $\hat{\Z} \cong \T = \R/\Z$, there is an isomorphism between the dual group $\hat{\mZ}$ and the group $\mT$, which we now describe.
Recall that, viewing $\F_q$ as an extension of $\F_p$, there is a trace map $\textup{Tr}_{\F_q/\F_p} : \F_q \to \F_p$, which can be written explicitly as $\textup{Tr}_{\F_q/\F_p}(c) = c + c^p + \ldots + c^{p^{k-1}}$, where $q = p^k$.
Every character on $\mZ$ is then of the form $u \mapsto e(u\alpha)$ for some $\alpha \in \mT$, where
\begin{equation*}
	e \left( \sum_{n=-\infty}^N c_n t^n \right) = \exp \left( \frac{2\pi i \cdot \textup{Tr}_{\F_q/\F_p}(c_{-1})}{p} \right).
\end{equation*}

In order to define a notion of uniform distribution in this context, we need an averaging method to replace averaging the first $N$ terms of a sequence and an analogue of subintervals of $[0,1]$.
The averaging method comes through the notion of a F{\o}lner sequence.
A \emph{F{\o}lner sequence} in $\mZ$ is a sequence $(F_N)_{N \in \N}$ of finite subsets of $\mZ$ such that
\begin{equation*}
	\frac{|F_N \cap (F_N + u)|}{|F_N|} \to 1
\end{equation*}
as $N \to \infty$.
The topology on $\mT$ is generated by cylinder sets
\begin{equation*}
	C(M; a_1, \ldots, a_M) = \left\{ \sum_{n=-\infty}^{-1} c_n t^n : c_{-1} = a_1, \ldots, c_{-M} = a_M \right\}
\end{equation*}
for $M \in \N$ and $a_1, \ldots, a_M \in \F_q$, and these will play the role of intervals.
Hence, we say that a $\mZ$-sequence $(a(u))_{u \in \mZ}$ is \emph{uniformly distributed} in $\mT$ with respect to a F{\o}lner sequence $(F_N)_{N \in \N}$ if for every $M \in \N$ and $a_1, \ldots, a_M \in \F_q$,
\begin{equation*}
	\frac{\left| \left\{ u \in F_N : a(u) \in C(M; a_1, \ldots, a_M) \right\} \right|}{|F_N|} \to q^{-M}.
\end{equation*}
We say that $(a(u))_{u \in \mZ}$ is \emph{well-distributed} if it is uniformly distributed with respect to every F{\o}lner sequence.

Contrasting with the situation in characteristic zero, there are polynomial sequences in the function field setting that have irrational coefficients but fail to be well-distributed.
For example, if $\alpha = \sum_{n=1}^{\infty} \alpha_n t^{-pn}$ with $(\alpha_1, \alpha_2, \ldots)$ not eventually periodic, then $\alpha \notin \mQ$, but for any $u = \sum_{n=0}^N c_n t^n \in \mZ$,
\begin{equation*}
	\alpha u^p = \sum_{n=1}^{\infty} \alpha_n t^{-pn} \cdot \sum_{m=0}^N c_m^p t^{pm} = \sum_{n=1}^{\infty} \left( \sum_{m=0}^N \alpha_{n+m} c_m^p \right) t^{-pn} + \underbrace{\sum_{n=0}^{N-1} \left( \sum_{m=n+1}^N \alpha_{m-n} c_m^p \right) t^{pn}}_{\in \mZ}
\end{equation*}
so $\alpha u^p$ is never contained in the cylinder set $C(1; 1)$.
Similar examples can be constructed for polynomials of the form $\alpha u^p + \beta u$, for example, using the fact that both $u$ and $u^p$ are additive homomorphisms; see \cite[Equation (6.9)]{carlitz}.

L\^{e}, Liu, and Wooley conjectured that the above examples are representative of the general situation in the sense that the only obstacles to equidistribution are exponents divisible by $p$ (creating obstructions because the Frobenius map is a homomorphism) and ``interference'' between exponents when one is a multiple of the other by a power of $p$; see \cite[Conjecture 1.3]{llw}.

The main result of this paper is the following function field analogue of Theorem \ref{thm: Weyl}, confirming \cite[Conjecture 1.3]{llw}\footnote{There are some notational differences between the present paper and \cite{llw}. L\^{e}, Liu, and Wooley use $\K_{\infty}$ to denote the field $\mR$ and $\T$ for the group $\mT$. Their conjecture also asks for the weaker conclusion that $(P(u))_{u \in \mZ}$ is uniformly distributed with respect to the particular F{\o}lner sequence $\mathbb{G}_N = \{u \in \mZ : \deg{u} < N\}$, but it turns out that uniform distribution and well-distribution are equivalent for polynomial sequences, so this makes little difference.}.

\begin{theorem} \label{thm: main}
	Let $\mK$ be a finite set of positive integers.
	Suppose $\alpha_r \in \mR$ for $r \in \mK \cup \{0\}$, and let
	\begin{equation*}
		P(u) = \sum_{r \in \mK \cup \{0\}} \alpha_r u^r \in \mR[u].
	\end{equation*}
	Suppose that for some $k \in \mK$,
	\begin{itemize}
		\item	$p \nmid k$,
		\item	for any $v \in \N$, $p^vk \notin \mK$, and
		\item	$\alpha_k \notin \mQ$.
	\end{itemize}
	Then $(P(u))_{u \in \mZ}$ is well-distributed in $\mT$.
\end{theorem}

We prove Theorem \ref{thm: main} in Section \ref{sec: proof}.
Afterwards, in Section \ref{sec: multivariable}, we give an extension of Theorem \ref{thm: main} to polynomials of several variables and polynomials over general function fields.

\begin{remark}
	A proof of Theorem \ref{thm: main} by different methods appears in \cite{cgllw}.
\end{remark}

%%%%%%%%%%%%%%%%%%%%%%%%%%%%%%%%%%%%%%%%%%%%%%%%%%%%%%%%
% ---- PROOF ---- %
%%%%%%%%%%%%%%%%%%%%%%%%%%%%%%%%%%%%%%%%%%%%%%%%%%%%%%%%

\section{Proof of Theorem \ref{thm: main}} \label{sec: proof}

The main ingredient in the proof of Theorem \ref{thm: main} is a general equidistribution theorem of Bergelson and Leibman that we reproduce below in Theorem \ref{thm: bl ud} using the notation established in the introduction.
Before stating their theorem, we need a few more definitions.

We say that a polynomial $\eta(u) \in \mR[u]$ is \emph{additive} if $\eta(u+v) = \eta(u) + \eta(v)$ for $u,v \in \mZ$.
Additive polynomials are precisely the polynomials of the form $\eta(u) = \alpha_0 u + \alpha_1 u^p + \ldots + \alpha_s u^{p^s}$ for $s \ge 0$ and $\alpha_0, \alpha_1, \ldots, \alpha_s \in \mR$.
We say that a monomial $u^r$ is \emph{separable} if $p \nmid r$ and a polynomial $P(u) \in \mR[u]$ is \emph{separable} if it is a linear combination of separable monomials.
By dividing out powers of $p$ in each term, every polynomial $P(u) \in \mR[u]$ can be uniquely expressed in the form $P(u) = \alpha_0 + \sum_{i=1}^d \eta_i(u^{r_i})$, where $\eta_1, \ldots, \eta_d$ are additive and $u^{r_1}, \ldots, u^{r_d}$ are distinct and separable.
The equidistribution theorem of Bergelson and Leibman says that the distributional behavior of $P(u)$ is controlled by the additive polynomials $\eta_1, \ldots, \eta_d$ involved in this decomposition.

\begin{theorem}[\cite{bl}, Theorem 0.3] \label{thm: bl ud}
	Any additive polynomial $\mZ$-sequence $(\eta(u))_{u \in \mZ}$ in $\mT$ is well-distributed in a set of the form $\mathcal{F}(\eta) + \eta(K)$, where $\mathcal{F}(\eta)$ is a $\Phi$-subtorus\footnote{For the purposes of this paper, it is enough to know that a $\Phi$-subtorus is a closed subgroup that can be associated to an additive polynomial $\eta$. For a proper definition, see \cite[Section 1]{bl}.} of $\mT$ and $K$ is a finite subset of $\mZ$.
	
	For any polynomial $\mZ$-sequence
	\begin{align*}
		P(u) = \alpha_0 + \sum_{i=1}^d{\eta_i \left( u^{r_i} \right)},
	\end{align*}
	where $\alpha_0 \in \mT$, $\eta_1, \dots, \eta_d$ are additive polynomial $\mZ$-sequences, and $u^{r_1}, \dots, u^{r_d}$ are distinct separable monomials, the closure $\mathcal{O}(P) = \overline{P(\mZ)}$ of $P(\mZ)$ has the form $\mathcal{F}(P) + P(K)$, where $\mathcal{F}(P)$ is the $\Phi$-subtorus $\sum_{i=1}^d{\mathcal{F}(\eta_i)}$ and $K$ is a finite subset of $\mZ$.
	
	Moreover, $(P(u))_{u \in \mZ}$ is well-distributed in the components of $\mathcal{F}(P) + P(k)$, $k \in K$, of $\mathcal{O}(P)$.
	In particular, $(P(u))_{u \in \mZ}$ is well-distributed in $\mT$ if $\sum_{i=1}^d{\mathcal{F}(\eta_i)} = \mT$.
\end{theorem}

\begin{remark}
	There is a subtle difference between the formulation of \cite[Theorem 0.3]{bl} as given in Theorem \ref{thm: bl ud} and the original formulation from \cite{bl}.
	In the present paper, we write that $\sum_{i=1}^d{\mathcal{F}(\eta_i)} = \mT$ is a sufficient condition for $(P(u))_{u \in \mZ}$ to be well-distributed in $\mT$.
	In \cite{bl}, it is incorrectly claimed that this condition is also necessary, a claim that was duplicated in an earlier version of this article on \emph{arXiv}.
	(The word ``if'' in the final sentence of Theorem \ref{thm: bl ud} was previously the phrase ``if and only if'' and appears in \cite{bl} as ``iff.''
	The ``if'' direction follows from the prior part of the theorem statement, while the ``only if'' direction does not.)
	A consequence of the incorrect statement is distilled as Hypothesis 5.1 in \cite{cgllw}, where a counterexample is also produced.
	Since the present paper is concerned only with a sufficient condition for equidistribution, this small error does not impact our usage of Theorem \ref{thm: bl ud}.
	Similar comments apply to the multivariable equidistribution result in Theorem \ref{thm: bl multivariable}.
\end{remark}

In the proof of Theorem \ref{thm: main}, we will establish well-distribution in $\mT$ by using Theorem \ref{thm: bl ud} and showing that for one of the additive polynomials $\eta$ appearing in the decomposition of $P$, we have $\mathcal{F}(\eta) = \mT$.
The additional ingredient used for this argument is the following result for linear polynomials:

\begin{proposition}[{\cite[Theorem 0.1]{bl}}] \label{prop: irrational dense orbit}
	Let $\alpha \in \mR \setminus \mQ$.
	Then for the additive polynomial $\eta(u) = \alpha u$, one has $\mathcal{F}(\eta) = \mT$.
\end{proposition}

\begin{proof}[Proof of Theorem \ref{thm: main}]
	The assumption $p \nmid k$ means that $u^k$ is a separable monomial.
	The additional assumption $k \in \mK$ and $p^vk \notin \mK$ for $v \in \N$ implies that $P$ decomposes as
	\begin{align*}
		P(u) = \alpha_0 + \alpha_k u^k + \sum_{s \in \mK'}{\eta_s \left( u^s \right)},
	\end{align*}
	where $\mK' = \left\{ s \in \N : p \nmid s~\text{and}~p^vs \in \mK \setminus \{k\}~\text{for some $v \ge 0$} \right\}$, and for $s \in \mK'$, $\eta_s$ is the additive polynomial
	\begin{align*}
		\eta_s(u) = \sum_{\substack{v \ge 0, \\ p^vs \in \mK}} \alpha_{p^vs} u^{p^v}.
	\end{align*}
	Finally, since $\alpha_k$ is irrational, the additive polynomial
	\begin{align*}
		\eta_k(u) = \alpha_k u
	\end{align*}
	satisfies $\mathcal{F}(\eta_k) = \mT$ by Proposition \ref{prop: irrational dense orbit}.
	Hence, $\mathcal{F}(\eta_k) + \sum_{s \in \mK'}{\mathcal{F}(\eta_s)} = \mT$.
	By Theorem \ref{thm: bl ud}, it follows that $(P(u))_{u \in \mZ}$ is well-distributed in $\mT$.
\end{proof}

%%%%%%%%%%%%%%%%%%%%%%%%%%%%%%%%%%%%%%%%%%%%%%%%%%%%%%%%
% ---- EXTENSION ---- %
%%%%%%%%%%%%%%%%%%%%%%%%%%%%%%%%%%%%%%%%%%%%%%%%%%%%%%%%

\section{A multivariable extension of Theorem \ref{thm: main}} \label{sec: multivariable}

Theorem \ref{thm: main} can be extended to multivariable polynomials.
To aid with the formulation of the theorem, we introduce the notation $S_d(r) = \{\bm{s} = (s_1, \ldots, s_d) : s_1 + \ldots + s_d = r\}$.

\begin{theorem} \label{thm: multivariable}
	Let $d \in \N$.
	Let $\mK$ be a finite set of positive integers.
	Suppose $\alpha_{\bm{s}} \in \mR$ for $\bm{s} \in S_d(r)$, $r \in \mK \cup \{0\}$, and let
	\begin{equation*}
		P(u_1, \ldots, u_d) = \alpha_0 + \sum_{r \in \mK \cup \{0\}} \sum_{\bm{s} \in S_d(r)} \alpha_{\bm{s}} \prod_{j=1}^d u_j^{s_j} \in \mR[u_1, \ldots, u_d].
	\end{equation*}
	Suppose that for some $k \in \mK$ and $\bm{l} \in S_d(k)$,
	\begin{itemize}
		\item	$p \nmid k$,
		\item	for any $v \in \N$, $p^vk \notin \mK$, and
		\item	$\alpha_{\bm{l}} \notin \mQ$.
	\end{itemize}
	Then $(P(\bm{u}))_{\bm{u} \in \mZ^d}$ is well-distributed in $\mT$.
\end{theorem}

The main tool for the proof of Theorem \ref{thm: multivariable} is the multivariable generalization of Theorem \ref{thm: bl ud} stated below.

\begin{theorem}[{\cite[Theorem 10.1]{bl}}] \label{thm: bl multivariable}
	Suppose
	\begin{align*}
		P(u_1, \ldots, u_d) = \alpha_0 + \sum_{i=1}^m{\eta_i \left( \prod_{j=1}^d u_j^{r_{i,j}} \right)},
	\end{align*}
	where $\alpha_0 \in \mT$, $\eta_1, \dots, \eta_m$ are additive polynomials, and $\prod_{j=1}^d u_j^{r_{1,j}}, \dots, \prod_{j=1}^d u_j^{r_{m,j}}$ are distinct separable\footnote{In this context, \emph{separable} means that at least one of the exponents $r_{i,j}$ is not divisible by $p$.} monomials.
	Then the closure $\mathcal{O}(P) = \overline{P(\mZ^d)}$ of $P(\mZ^d)$ has the form $\mathcal{F}(P) + P(K)$, where $\mathcal{F}(P)$ is the $\Phi$-subtorus $\sum_{i=1}^m{\mathcal{F}(\eta_i)}$ and $K$ is a finite subset of $\mZ^d$.
	
	Moreover, $(P(u_1, \ldots, u_d))_{(u_1,\ldots,u_d) \in \mZ^d}$ is well-distributed in the components of $\mathcal{F}(P) + P(k)$, $k \in K$, of $\mathcal{O}(P)$.
	In particular, $(P(u_1, \ldots, u_d))_{(u_1,\ldots,u_d) \in \mZ^d}$ is well-distributed in $\mT$ if $\sum_{i=1}^m{\mathcal{F}(\eta_i)} = \mT$.
\end{theorem}

We can now prove Theorem \ref{thm: multivariable}.

\begin{proof}[Proof of Theorem \ref{thm: multivariable}]
	As in the proof of Theorem \ref{thm: main}, let
	\begin{equation*}
		\mK' = \left\{ s \in \N : p \nmid s~\text{and}~p^vs \in \mK \setminus \{k\}~\text{for some $v \ge 0$} \right\}.
	\end{equation*}
	Then $P(u_1, \ldots, u_d)$ decomposes as
	\begin{equation*}
		P(u_1, \ldots, u_d) = \alpha_0 + \sum_{\bm{l} \in S_d(k)} \alpha_{\bm{l}} \prod_{j=1}^d u_j^{l_j} + \sum_{s \in \mK'} \sum_{\bm{t} \in S_d(s)} \eta_{\bm{t}} \left( \prod_{j=1}^d u_j^{t_j} \right),
	\end{equation*}
	where
	\begin{equation*}
		\eta_{\bm{t}}(u) = \sum_{\substack{v \ge 0 \\ p^vs \in \mK}} \alpha_{p^v\bm{t}} \prod_{j=1}^d u_j^{p^vt_j}
	\end{equation*}
	for $\bm{t} \in S_d(s)$.
	
	For $\bm{l} \in S_d(k)$, let $\eta_{\bm{l}}(u) = \alpha_{\bm{l}} u$.
	By assumption, $\alpha_{\bm{l}} \notin \mQ$ for some $\bm{l} \in S_d(k)$, whence $\mathcal{F}(\eta_{\bm{l}}) = \mT$ by Proposition \ref{prop: irrational dense orbit}.
	Therefore, $(P(u_1, \ldots, u_d))_{(u_1,\ldots,u_d) \in \mZ^d}$ is well-distributed in $\mT$ by Theorem \ref{thm: bl multivariable}.
\end{proof}

A consequence of Theorem \ref{thm: multivariable} is an equidistribution theorem for polynomial sequences in the dual group of the ring of integers of an arbitrary global function field.
In this abstract setting, we use the following notation.
We use $\K$ to denote a global function field (i.e., a finite extension of $\F_q(t)$ for some $q$) and denote its ring of integers (the integral closure of $\F_q[t]$) by $\mZ_{\K}$.
We will write $\mT_{\K}$ for the dual group of $\mZ_{\K}$.
For us, polynomial sequences in $\mT_{\K}$ will be $\mZ_{\K}$-sequences of the form $P(u) = \prod_{j=0}^d \alpha_j^{u^j}$, where $\alpha_j \in \mZ_{\K}$, and $\alpha^u$ is the character defined by $\alpha^u : v \mapsto \alpha(uv)$.
When $\K = \F_q(t)$, this agrees with the notion of polynomial sequences dealt with in Theorem \ref{thm: main} up to notational changes.

The appropriate notion of irrationality for ``coefficients'' $\alpha \in \mT_{\K}$ is defined as follows: we say that $\alpha \in \mT_{\K}$ is \emph{rational} if $\alpha^u = 1$ for some $u \in \mZ_{\K}$ and \emph{irrational} otherwise.

\begin{corollary} \label{cor: global}
	Let $\K$ be a global function field, and let $\mZ_{\K}$ be its ring of integers and $\mT_{\K}$ the compact dual group of $\mZ_{\K}$.
	Let $\mK$ be a finite set of positive integers.
	Suppose $\alpha_r \in \mT_{\K}$ for $r \in \mK \cup \{0\}$, and let
	\begin{equation*}
		P(u) = \prod_{r \in \mK \cup \{0\}} \alpha_r^{u^r}.
	\end{equation*}
	Suppose that for some $k \in \mK$,
	\begin{itemize}
		\item	$p \nmid k$,
		\item	for any $v \in \N$, $p^vk \notin \mK$, and
		\item	$\alpha_k$ is irrational.
	\end{itemize}
	Then $(P(u))_{u \in \mZ_{\K}}$ is well-distributed in $\mT_{\K}$.
\end{corollary}

\begin{proof}
	Suppose $\K$ is a degree $d$ extension of $\F_q(t)$.
	Let $b_1, \ldots, b_d \in \mZ_{\K}$ be an integral basis.
	That is, $b_1, \ldots, b_d$ is a basis for $\K$ over $\F_q(t)$, and
	\begin{equation*}
		\mZ_{\K} = \left\{ \sum_{j=1}^d u_j b_j : u_j \in \F_q[t] \right\}.
	\end{equation*}
	For a proof of existence of such a basis, see \cite[Chapter IV, Theorem 3]{chevalley}.
	
	The integral basis $b_1, \ldots, b_d$ establishes an isomorphism $\mZ_{\K} \cong \mZ^d$, so $\mT_{\K}$ is isomorphic to $\mT^d$.
	In particular, for each $r \in \mK \cup \{0\}$, there exist $\alpha_{r,1}, \ldots, \alpha_{r,d} \in \mT$ such that
	\begin{equation*}
		\alpha_r \left( \sum_{j=1}^d u_j b_j \right) = e \left( \sum_{j=1}^d \alpha_{r,j} u_j \right)
	\end{equation*}
	for $(u_1, \ldots, u_d) \in \mZ^d$.
	
	For $s \in \N$, $j_1, \ldots, j_s \in \{1, \ldots, d\}$, and $l \in \{1, \ldots, d\}$, let $c_{j_1, \ldots, j_s, l} \in \mZ$ such that $\prod_{i=1}^s b_{j_i} = \sum_{l=1}^d c_{j_1, \ldots, j_s,l} b_l$.
	Then given $u = \sum_{j=1}^d u_j b_j \in \mZ_{\K}$ and $r \in \N$, we have
	\begin{equation*}
		u^r = \left( \sum_{j=1}^d u_j b_j \right)^r = \sum_{j_1, \ldots, j_r=1}^d u_{j_1} \cdot \ldots \cdot u_{j_r} b_{j_1} \cdot \ldots \cdot b_{j_r} = \sum_{j_1, \ldots, j_r, l=1}^d c_{j_1, \ldots, j_r, l} u_{j_1} \cdot \ldots \cdot u_{j_r} b_l.
	\end{equation*}
	Let $P_{r,l}(u_1, \ldots, u_d) = \sum_{j_1,\ldots,j_r=1}^d c_{j_1, \ldots, j_r, l} u_{j_1} \cdot \ldots \cdot u_{j_r} \in \mZ[u_1, \ldots, u_d]$ so that
	\begin{equation*}
		u^r = \sum_{l=1}^d P_{r,l}(u_1, \ldots, u_d) b_l.
	\end{equation*}
	Note that $P_{r,l}$ is homogeneous of degree $r$.
	Moreover, if $p \nmid r$, then the polynomials $P_{r,1}, \ldots, P_{r,d}$ are linearly independent (in fact, algebraically independent) over $\mQ$ by \cite[Lemma 7.5]{ab}.
	
	In order to establish well-distribution of $(P(u))_{u \in \mZ_{\K}}$ in $\mT_{\K}$, it suffices by Weyl's criterion to prove that $(\chi(P(u)))_{u \in \mZ_{\K}}$ is well-distributed in the group of $p$th roots of unity for every nontrivial character $\chi : \mT_{\K} \to S^1$.
	Fix a nontrivial character $\chi : \mT_{\K} \to S^1$.
	Then there exists a nonzero $\bm{v} = (v_1, \ldots, v_d) \in \mZ^d$ such that for all $r \in \mK \cup \{0\}$,
	\begin{equation*}
		\chi(\alpha_r) = e \left( \sum_{j=1}^d \alpha_{r,j} v_j \right).
	\end{equation*}
	We then have
	\begin{equation*}
		\chi(P(u)) = \chi \left( \prod_{r \in \mK \cup \{0\}} \alpha_r^{u^r} \right) = e \left( \sum_{r \in \mK \cup \{0\}} \sum_{j,l,m=1}^d c_{j,l,m} \alpha_{r,j} P_{r,l}(u_1, \ldots, u_d) v_m \right)
	\end{equation*}
	for $u = \sum_{j=1}^d u_j b_j \in \mZ_{\K}$.
	We will use Theorem \ref{thm: multivariable} to show that the polynomial
	\begin{equation*}
		Q(u_1, \ldots, u_d) = \sum_{r \in \mK \cup \{0\}} \sum_{j,l,m=1}^d c_{j,l,m} \alpha_{r,j} P_{r,l}(u_1, \ldots, u_d) v_m
	\end{equation*}
	is well-distributed in $\mT$.
	For $r \in \mK \cup \{0\}$ and $j \in \{1, \ldots, d\}$, let
	\begin{equation*}
		Q_{r,j}(u_1, \ldots, u_d) = \sum_{l,m=1}^d c_{j,l,m} v_m P_{r,l}(u_1, \ldots, u_d)
	\end{equation*}
	so that $Q = \sum_{r \in \mK \cup \{0\}} \sum_{j=1}^d \alpha_{r,j} Q_{r,j}$.
	Note that $Q_{r,j}$ is homogeneous of degree $r$.
	
	We claim that $Q_{k,1}, \ldots, Q_{k,d}$ are linearly independent.
	Indeed, suppose for contradiction that $\bm{a} = (a_1, \ldots, a_d) \in \mZ^d \setminus \{\bm{0}\}$ and
	\begin{equation*}
		\sum_{j=1}^d a_j Q_{k,j} = 0.
	\end{equation*}
	Then substituting in the definition of $Q_{k,j}$, we have
	\begin{equation*}
		\sum_{j,l,m=1}^d a_j c_{j,l,m} v_m P_{k,l} = 0.
	\end{equation*}
	But since $p \nmid k$, the polynomials $P_{k,1}, \ldots, P_{k,d}$ are linearly independent, so we conclude
	\begin{equation} \label{eq: e_l=0}
		e_l = \sum_{j,m=1}^d a_j c_{j,l,m} v_m = 0
	\end{equation}
	for every $l \in \{1, \ldots, d\}$.
	From the definition of the coefficients $c_{j,l,m}$, we have the symmetry $c_{j,l,m} = c_{l,j,m}$, so
	\begin{equation*}
		e_l = \sum_{j,j=1}^d a_j c_{l,j,m} v_m = 0
	\end{equation*}
	for every $l \in \{1, \ldots, d\}$.
	We can rewrite the formula for $e_l$ using the matrix
	\begin{equation*}
		M_l = \left( \begin{array}{ccc}
			c_{l,1,1} & \ldots & c_{l,d,1} \\
			\vdots & \ddots & \vdots \\
			c_{l,1,d} & \ldots & c_{l,d,d}
		\end{array} \right)
	\end{equation*}
	as
	\begin{equation*}
		e_l = \bm{v} M_l \bm{a}^T.
	\end{equation*}
	But the matrix $M_l$ represents multiplication by $b_l$ in the basis $\{b_1, \ldots, b_d\}$, and the vectors $\bm{v} = (v_1, \ldots, v_d)$ and $\bm{a} = (a_1, \ldots, a_d)$ are both nonzero.
	Let
	\begin{equation*}
		x = \frac{\sum_{j=1}^d v_j b_j}{\sum_{j=1}^d a_j b_j} \in \K.
	\end{equation*}
	Writing $x = \sum_{j=1}^d x_j b_j$ for some $(x_1, \ldots, x_d) \in \mQ^d \setminus \{\bm{0}\}$, we have that $M = \sum_{l=1}^d x_l M_l$ represents multiplication by $x$ in the basis $\{b_1, \ldots, b_d\}$.
	In particular, $M\bm{a}^T = \bm{v}^T$.
	But then $\sum_{l=1}^d x_l e_l = \bm{v}M\bm{a}^T = \bm{v} \bm{v}^T \ne 0$, which contradicts \eqref{eq: e_l=0}.
	We conclude that there can be no nontrivial linear combination of $Q_{k,1}, \ldots, Q_{k,d}$ that is equal to 0.
	That is, $Q_{k,1}, \ldots, Q_{k,d}$ are linearly independent as claimed.
	
	Now, since $\alpha_k$ is irrational, at least one of the coordinates $\alpha_{k,1}, \ldots, \alpha_{k,d}$ is irrational (i.e., not an element of $\mQ/\mZ$) by \cite[Lemma 7.4]{ab}.
	Combining this observation with the linear independence of $Q_{k,1}, \ldots, Q_{k,d}$, it follows that if we write
	\begin{equation*}
		\sum_{j=1}^d \alpha_{k,j} Q_{k,j}(u_1, \ldots, u_d) = \sum_{\bm{l} \in S_d(k)} \beta_{\bm{l}} \prod_{i=1}^d u_i^{l_i},
	\end{equation*}
	then at least one of the coefficients $\beta_{\bm{l}}$ is irrational.
	Thus, $Q$ satisfies the hypothesis of Theorem \ref{thm: multivariable}, so $(Q(u_1,\ldots, u_d))_{(u_1, \ldots, u_d) \in \mZ^d}$ is well-distributed in $\mT$.
	Consequently, $\chi(P(u)) = e(Q(u_1, \ldots, u_d))$ is well-distributed in the group of $p$th roots of unity as desired.
\end{proof}

%%%%%%%%%%%%%%%%%%%%%%%%%%%%%%%%%%%%%%%%%%%%%%%%%%%%%%%%
% ---- ACKNOWLEDGEMENTS ---- %
%%%%%%%%%%%%%%%%%%%%%%%%%%%%%%%%%%%%%%%%%%%%%%%%%%%%%%%%

\section*{Acknowledgements}

This work is supported by the Swiss National Science Foundation under grant TMSGI2-211214.

%%%%%%%%%%%%%%%%%%%%%%%%%%%%%%%%%%%%%%%%%%%%%%%%%%%%%%%%
% ---- REFERENCES---- %
%%%%%%%%%%%%%%%%%%%%%%%%%%%%%%%%%%%%%%%%%%%%%%%%%%%%%%%%

\end{document}